\documentclass[10pt,a4paper]{article}
\usepackage[utf8]{inputenc}
\usepackage{amsmath}
\usepackage{amsfonts}
\usepackage{amssymb}
\usepackage{amsthm}
\usepackage{enumerate}
\usepackage{marvosym}
\usepackage{hyperref}
 \newtheorem{thm}{Theorem}[section]

 \newtheorem*{thm*}{Theorem}
 \newtheorem{cor}[thm]{Corollary}
 \newtheorem{lem}[thm]{Lemma}
 \newtheorem{prop}[thm]{Proposition}
 \theoremstyle{definition}
 
 \theoremstyle{remark}

 \newtheorem{quest}{Question}

 \numberwithin{equation}{section}

\newcommand{\tr}{\operatorname{Tr}}

\newcommand{\im}{\operatorname{Im}}

\begin{document}
\title{A Wiener algebra for Fock space operators}
\author{Robert Fulsche}
\date{\today}

\maketitle
%%%%%%%%%%%%%%%%%%%%%%%%%%%%%%%%%%%%%%%%%%%%%%%%%%%%%%%%%%%%%%%%%%%%%%%%%%%%%%%%%%%%%%%%%%%%%%
\begin{abstract}
    We introduce an algebra $\mathcal W_t$ of linear operators that act continuously on each of the Fock spaces $F_t^p$, $1 \leq p \leq \infty$, and contains all Toeplitz operators with bounded symbols. We show that compactness, the spectrum, essential spectrum and the Fredholm index of an element of $\mathcal W_t$, realized as an operator on $F_t^p$, are independent of the value of $p$.
\end{abstract}

\section{Introduction}

Spectral theory for operators on Fock spaces $F_t^p$ (also known as Segal-Bargmann spaces) has received steady interest over the past years. Among the results achieved are statements about compactness \cite{Bauer_Fulsche2019, Bauer_Isralowitz2012, Isralowitz_Mitkovski_Wick, Tien_Khoi2019} and Fredholm properties \cite{Al-Qabani_Virtanen, Fulsche_Hagger, Fulsche2022}. Toeplitz operators are certainly among the most prominent operators studied on these spaces, and we will restrict our introductory discussion to them. 

The first thing one can observe is that a sum of products of Toeplitz operators, which we simply shall denote by $A$, yields a bounded linear operator on the Fock space $F_t^p$ for every $p \in [1, \infty]$. It was proven in \cite{Englis1999} (for $p = 2$) and later in \cite{Bauer_Isralowitz2012} (for $p \in (1, \infty)$) that such a finite sum of products of Toeplitz operators on $F_t^p$ is compact if and only if the Berezin transform of the operator vanishes at infinity. Since the Berezin transform of $A$ is given by $\widetilde{A}(z) = \langle A k_z^t, k_z^t\rangle$, where $k_z^t$ is the normalized reproducing kernel of $F_t^2$ and the inner product is the dual pairing induced from $F_t^2$, one sees that the property $\widetilde{A} \in C_0$ is independent of the particular Fock space $F_t^p$ on which we realize the operator $A$, i.e.\ compactness is independent of the value of $p$.

In \cite{Al-Qabani_Virtanen} it is proven that the essential spectrum of a Toeplitz operator $T_f$ with symbol $f \in \operatorname{VO}(\mathbb C^n)$ (i.e.\ of vanishing oscillation) is given by $\sigma_{ess}(T_f) = \bigcap_{r > 0} (f(B(0,r)^c))$, the values of $f$ at infinity, and this is true on $F_t^p$ for all $p \in (1, \infty)$. Together with the above-mentioned compactness result, these are about all explicit results that are known regarding the spectrum of a Toeplitz operator on a Fock space. A common fact about both results is that the spectral data obtained is independent of the value of $p$. In this work, we will investigate this phenomenon in detail.

To be more precise, we will introduce a class of linear operators $\mathcal W_t$ that acts on each of the Fock $F_t^p$ (with $t > 0$ fixed and $p$ varying) continuously. The class will be defined by certain off-diagonal properties of the operator's integral kernels, details will be given in the main part of the paper. Since $\mathcal W_t$ is an algebra containing all Toeplitz operators with bounded symbols, there is a large class of non-trivial operators in $\mathcal W_t$. Our main results can briefly be summarized as follows:
\begin{thm*}
    Let $A \in \mathcal W_t$. \begin{enumerate}
        \item $A$ acts as a bounded operator on each of the Fock spaces $F_t^p$, $p \in [1, \infty]$. Further, $A$ is contained in the Banach algebra generated by all Toeplitz operators with bounded symbols.
        \item Compactness of $A$ is independent of the value of $p \in [1, \infty]$.
        \item $\sigma(A)$ and $\sigma_{ess}(A)$ are both independent of $p \in [1, \infty]$.
        \item If $A$ is Fredholm, then $\operatorname{ind}(A)$ is independent of $p$.
    \end{enumerate}
\end{thm*}
While it seems that the algebra $\mathcal W_t$ has so far never been considered in the literature on Fock space operator theory, we want to emphasize that we actually do not give any new deep proofs in our work. Our findings are essentially based on several deep results that have been obtained in the last years. Our contribution here is that we provide the right definitions and draw connections between several established theorems. Nevertheless, we think that the theme of spectral independence on Fock spaces is an interesting field which so far has not been properly studied. The present work should be seen as a starting point, connection known results under the present point of view and hopefully motivating more research in this direction.

\section{Preliminaries}
On $\mathbb C^n$ we consider the family of probability measures $\mu_t$ given by
\begin{align*}
d\mu_t(z) = \frac{1}{(\pi t)^n} e^{-\frac{|z|^2}{t}}~dz,
\end{align*}
where $| \cdot |$ is the Euclidean norm, $dz$ the standard Lebesgue measure and $t > 0$ a fixed real number. Given any $1 \leq p < \infty$, we define the Fock space $F_t^p$ by
\begin{align*}
F_t^p = L_t^p \cap \operatorname{Hol}(\mathbb C^n),
\end{align*}
where $\operatorname{Hol}(\mathbb C^n)$ denotes the entire functions on $\mathbb C^n$ and $L_t^p := L^p(\mathbb C^n, \mu_{2t/p})$. This space is always endowed with its natural $L^p$-norm, i.e.
\begin{align*}
\| f\|_p^p = \left( \frac{p}{2\pi t} \right)^n \int_{\mathbb C^n} |f(z)|^p e^{-\frac{p}{2t}|z|^2}~dz.
\end{align*}
These spaces are well-known to be closed subspaces of $L^p(\mathbb C^n, \mu_{2t/p})$.

For $p = \infty$, two different Fock spaces come into play: $F_t^\infty$ consists of all entire functions $f$ on $\mathbb C^n$ such that
\begin{align*}
\| f\|_{F_t^\infty} := \sup_{z \in \mathbb C^n} |f(z)|e^{-\frac{|z|^2}{2t}} < \infty.
\end{align*}
Further, we also consider
\begin{align*}
f_t^\infty := \{ f \in F_t^\infty: fe^{-\frac{|\cdot|^2}{2t}} \in C_0(\mathbb C^n)\}.
\end{align*}
Both spaces are well-known to be complete, $f_t^\infty$ being a closed subspace of $F_t^\infty$. The standard references for this family of Fock spaces are \cite{Janson_Peetre_Rochberg1987, Zhu2012}.

For $p = 2$, we clearly are in the Hilbert space setting. As is well-known, $F_t^2$ is even a reproducing kernel Hilbert space, the reproducing kernels being given by
\begin{align*}
K_z^t(w) = e^{\frac{w \cdot \overline{z}}{t}}.
\end{align*}
Indeed, the functions $K_z$ are contained in any of the spaces $F_t^p$ and $f_t^\infty$, and they span a dense subspace of $F_t^p$ for $1 \leq p < \infty$ and of $f_t^\infty$. The normalized reproducing kernels are now defined as
\begin{align*}
k_z^t(w) = \frac{K_z^t(w)}{\| K_z^t\|_{F_t^2}} = e^{\frac{w \cdot \overline{z}}{t} - \frac{|z|^2}{t}}.
\end{align*}
It follows from elementary computations that we indeed have $\| k_z^t\|_{F_t^p} = 1$ for any $p$.

Under the $F_t^2$ inner product, the Fock spaces satisfy the following duality relations:
\begin{itemize}
\item For $1 \leq p < \infty$: $(F_t^p)' \cong F_t^q$, where $\frac{1}{p} + \frac{1}{q} = 1$;
\item $(f_t^\infty) \cong F_t^1$;
\item $(F_t^\infty)'$ strictly contains $F_t^1$.
\end{itemize}
Besides the duality structure of the Fock spaces, we will also make use of their behaviour under the complex interpolation method. As the standard textbook reference for interpolation methods we refer to \cite{Bergh_Lofstrom1976}. A couple of Banach spaces $\overline{X} = (X_0, X_1)$ is said to be compatible if there exists a Hausdorff topological vector space $X$ such that both $X_0$ and $X_1$ are subspaces of $X$ and the embeddings $X_j \hookrightarrow X$ are continuous. For such a compatible couple, we let $\Delta(\overline{X}) = X_0 \cap X_1$ and $\Sigma(\overline{X}) = X_0 + X_1$. We denote the interpolation space obtained by the complex interpolation method with parameter $\theta \in [0, 1]$ by $\overline{X}_{[\theta]}$. Further, we say that $A \in \mathcal L(\overline{X})$ if $A: X_0 + X_1 \to X_0 + X_1$ and $A|_{X_0}: X_0 \to X_0$, $A|_{X_1}: X_1 \to X_1$ continuously. Then, $A: (X_0, X_1)_{[\theta]} \to (X_0, X_1)_{[\theta]}$ continuously for every $\theta \in [0, 1]$.

The couple $\overline{X} = (F_t^1, f_t^\infty)$ is a compatible couple with ambient space $X = f_t^\infty$. Then, $\Delta(\overline{X}) = F_t^1$ and $\Sigma(\overline{X}) = f_t^\infty$ by the usual embedding properties of Fock spaces.

Zhu proves in \cite[Theorem 2.29]{Zhu2012} that $[F_t^1, f_t^\infty]_{[\theta]} = F_t^p$ with $p = \frac{1}{1-\theta}$ for every $\theta \in (0, 1)$ (more precisely, he proves $[F_t^1, F_t^\infty]_{[\theta]} = F_t^p$; an application of \cite[Theorem 4.2.2(b)]{Bergh_Lofstrom1976} then yields the statement). Further, since $\Delta(\overline{X}) = F_t^1$ is dense in $F_t^1$ and $f_t^\infty$, \cite[Theorem 4.2.2 (a) and (c)]{Bergh_Lofstrom1976} show that $[F_t^1, f_t^\infty]_{[0]} = F_t^1$, $[F_t^1, f_t^\infty]_{[1]} = f_t^\infty$. Similarly, $[L_t^1, L_t^\infty]_{[\theta]} = L_t^p$ for $p = \frac{1}{1-\theta}$ and $\theta \in (0, 1)$.

We will consider $\mathcal L(X)$, the bounded linear operators on $X$, for $X$ any of the above Fock spaces. For $p = 2$, we of course have the well-known orthogonal projection $P_t \in \mathcal L(L^2(\mathbb C^n, \mu_t))$ mapping onto $F_t^2$ by
\begin{align*}
P_t f(z) = \langle f, K_z^t\rangle = \frac{1}{(\pi t)^n} \int_{\mathbb C^n} f(w) e^{\frac{z \cdot \overline{w}}{t}} e^{-\frac{|w|^2}{t}}~dw.
\end{align*}
As is well-known, $P_t$ (interpreted as an integral operator) gives rise to a bounded projection on $L^p(\mathbb C^n, \mu_{2t/p})$ mapping onto $F_t^p$ for any $1 \leq p \leq \infty$. Hence, for any $f \in L^\infty(\mathbb C^n)$ and $1 \leq p \leq \infty$ the Toeplitz operator $T_f^t$ given by
\begin{align*}
T_f^t: F_t^p \to F_t^p, \quad T_f^t(g) = P_t (fg)
\end{align*}
is well defined and bounded, satisfying $\| T_f^t\|_{F_t^p \to F_t^p} \leq C_{n, p, t} \| f\|_\infty$. Further, every Toeplitz operator leaves $f_t^\infty$ invariant. Therefore, we can also consider $T_f^t \in \mathcal L(f_t^\infty)$.

For $z \in \mathbb C^n$ we define the Weyl operator $W_z^t$ by
\begin{align*}
W_z^tg(w) = k_z^t(w) g(w-z). 
\end{align*}
Indeed, $W_z^t = T_{g_z}^t$ with
\begin{align*}
g_z^t(w) = e^{\frac{|z|^2}{2t} + \frac{2i \operatorname{Im}(w \cdot \overline{z})}{t}},
\end{align*}
in particular $W_z^t \in \mathcal L(F_t^p)$ for every $1 \leq p \leq \infty$ and $W_z^t \in \mathcal L(f_t^\infty)$.
These operators satisfy the following well-known properties, which we fix as a lemma:
\begin{lem}\label{lem:prop_weyl}
\begin{enumerate}[(1)]
\item $W_z^t$ is an isometry on $F_t^p$ for any $1 \leq p \leq \infty$;
\item $W_z^t W_w^t = e^{-\frac{\operatorname{Im}(z \cdot \overline w)}{t}}W_{z+w}^t$ for any $z, w \in \mathbb C^n$; 
\item For $1 \leq p < \infty$, $z \mapsto W_z^t$ is continuous in strong operator topology over $F_t^p$. The same holds true over $f_t^\infty$.
\end{enumerate}
\end{lem}
Note that (1) and (2) above follow from elementary computations, while (3) holds for $p < \infty$ by Scheff\'{e}'s lemma, and the case of $f_t^\infty$ can be proven by some standard $\varepsilon$-$\delta$ argument.

The main goal of this paper is the investigation of properties of linear operators which are defined on each of those Fock spaces. For doing so, we will usually denote $\mathbb F_t \in \{ F_t^p; ~1 \leq p \leq \infty\} \cup \{ f_t^\infty\}$ to denote any of the Fock spaces (with respect to the fixed parameter $t$).

Note that (2) of Lemma \ref{lem:prop_weyl} can also be written as $W_z^t W_w^t = e^{-i\sigma_t(z, w)}$, where $\sigma_t(z, w) = \im(z \cdot \overline{w})/t$ is a symplectic form on the 2n-dimensional real vector space $\mathbb C^n \cong \mathbb R^{2n}$. Since the projective unitary representation $z \mapsto W_z^t$ of $\mathbb C^n$ on $F_t^2$ is irreducible (we have $W_z^t 1 = K_z^t$, hence $1$ is a cyclic vector), the theorem of Stone and von Neumann asserts that there is a unitary map $\mathcal B_t: L^2(\mathbb R^n) \to F_t^2$ satisfying $\mathcal B_t V_{(x, \xi)}^t = W_{x + i\xi} \mathcal B_t$, where
\begin{align*}
    V_{(x, \xi)}^tf(s) = e^{i\frac{x\cdot \xi}{2t} - i\frac{s\cdot \xi}{t}}f(s-x), \quad x, \xi, s \in \mathbb R^n.
\end{align*}
$\mathcal B_t$ is unique up to multiplication with a complex constant of absolute value one. Upon choosing this constant correctly, $\mathcal B_t$ is the \emph{Bargmann transform}. We will not give the precise formula for $\mathcal B_t$ and refer to the literature instead \cite{Folland1989, Zhu2012}. What is more important to us is the fact that the Bargmann transform extends to an isometric and surjective map $M^p(\mathbb R^n) \to F_t^p(\mathbb C^n)$, cf.\ \cite{Signahl_Toft2011} and references therein. Here, $M^p(\mathbb R^n)$ denotes the $p$-modulation space. The precise definition of this space is not of importance to us and can be found in the references.

\section{The Fock space Wiener algebra}
Every bounded linear operator on $\mathbb F_t$, $\mathbb F_t \neq F_t^\infty$, is an integral operator. This can be seen as follows: Given $A \in \mathcal L(F_t^p)$ we have
\begin{align*}
Af(z) &= \langle Af, K_z^t\rangle = \langle f, A^\ast K_z^t \rangle\\
&= \int_{\mathbb C^n} f(w) \overline{A^\ast K_z^t(w)}~d\mu_t(w)\\
&= \int_{\mathbb C^n} f(w) \langle AK_w^t, K_z^t\rangle ~d\mu_t(w).
\end{align*}
The same holds true for $A \in \mathcal L( F_t^\infty)$ provided it leaves $f_t^\infty$ invariant (cf. \cite{Fulsche2022}).
The bivariate Berezin transform is given by
\begin{align*}
\widetilde{A}(w,z) = \langle Ak_w^t, k_z^t\rangle.
\end{align*}
The map $A \mapsto \widetilde{A}$ is injective (for $\mathbb F_t \neq F_t^\infty$). The integral kernel of an operator can be written in terms of the bivariate Berezin transform:
\begin{align*}
Af(z) &= \int f(w) \widetilde{A}(w,z) e^{\frac{|w|^2+|z|^2}{2t}}~d\mu_t(w)\\
&= \int f(w) \langle A K_w^t, K_z^t\rangle ~d\mu_t(w).
\end{align*}
Recall that we want to study properties of operators acting on each of the Fock spaces. By this, we will now understand integral operators which define a bounded linear operator for each choice of $\mathbb F_t$. Of course, we will try to deduce properties of such operators from their integral kernels. Before doing so, we will need a suitable criterion for the boundedness of such integral operators.
\begin{lem}\label{normestimate}
Let $k: \mathbb C^n \times \mathbb C^n \to \mathbb C$ be measurable such that
\begin{align*}
A_1 &:= \sup_{w \in \mathbb C^n} \int_{\mathbb C^n} |k(w,z)| e^{-\frac{|z|^2 + |w|^2}{2t}}~dz < \infty,\\
A_\infty &:= \sup_{z \in \mathbb C^n} \int_{\mathbb C^n} |k(w,z)| e^{-\frac{|z|^2 + |w|^2}{2t}}~dw < \infty. 
\end{align*}
Then, the integral operator
\begin{align*}
Af(z) = \int_{\mathbb C^n} f(w) k(w,z)~d\mu_t(w)
\end{align*}
is bounded on $L_t^p$ for any $p \in [1, \infty]$ with
\begin{align*}
\| A\|_{L_t^p \to L_t^p} \leq \left( \frac{1}{\pi t} \right)^n A_1^{1-\theta} A_\infty^\theta,
\end{align*}
where $\theta = 1 - \frac{1}{p}$.
\end{lem}
\begin{proof}
For $p = 1$, we obtain:
\begin{align*}
\| Af\|_{L_t^1} &= \left( \frac{1}{2\pi t}\right)^n \int_{\mathbb C^n} |Af(z)|e^{-\frac{|z|^2}{2t}}~dz\\
&\leq \left( \frac{1}{2\pi t} \right)^n \left( \frac{1}{\pi t} \right)^n \int_{\mathbb C^n} \int_{\mathbb C^n} |f(w)| |k(w,z)|e^{-\frac{|w|^2}{t} - \frac{|z|^2}{2t}}~dw~dz\\
&\leq \left( \frac{1}{2\pi t} \right)^n \left( \frac{1}{\pi t} \right)^n \int_{\mathbb C^n}|f(w)| e^{-\frac{|w|^2}{2t}} \int_{\mathbb C^n} |k(w,z)|e^{-\frac{|w|^2 + |z|^2}{2t}}~dz~dw\\
&\leq \left( \frac{1}{2\pi t} \right)^n \left( \frac{1}{\pi t} \right)^n \int_{\mathbb C^n}|f(w)| e^{-\frac{|w|^2}{2t}}dw \sup_{v \in \mathbb C^n} \int_{\mathbb C^n} |k(v,z)|e^{-\frac{|v|^2+|z|^2}{t}}~dz\\
&= \left( \frac{1}{\pi t} \right)^n A_1 \| f\|_{L_t^1}.
\end{align*}
For $p = \infty$ we get:
\begin{align*}
\| Af\|_{L_t^\infty} &= \sup_{z \in \mathbb C^n} |Af(z)|e^{-\frac{|z|^2}{2t}}\\
&\leq \sup_{z \in \mathbb C^n} \left( \frac{1}{\pi t} \right)^n \int_{\mathbb C^n} |f(w)| |k(w,z)|e^{-\frac{|w|^2}{t} - \frac{|z|^2}{2t}}~dw\\
&\leq \left( \frac{1}{\pi t} \right)^n \| f\|_{F_t^\infty} \sup_{z \in \mathbb C^n} \int_{\mathbb C^n} |k(w,z)|e^{-\frac{|z|^2 + |w|^2}{2t}}~dw\\
&= \left( \frac{1}{\pi t} \right)^n \| f\|_{L_t^\infty} A_\infty.
\end{align*}
Complex interpolation yields the result for any $p$.
\end{proof}
We now consider a special class of integral operators: We denote by $\mathcal W_t$ the class of all measurable $k: \mathbb C^n \times \mathbb C^n \to \mathbb C$ such that:
\begin{enumerate}
\item $k(w,z)$ is holomorphic in $z$ and anti-holomorphic in $w$;
\item There exist some non-negative $H \in L^1(\mathbb C^n)$ such that
\begin{equation}\label{defa0}
|k(w,z)| \leq e^{\frac{|z|^2 + |w|^2}{2t}} H(w-z).
\end{equation}
\end{enumerate}
On $\mathcal W_t$ we define the norm
\begin{align*}
\| k\|_{\mathcal W_t} := \frac{1}{(\pi t)^n}\inf \{ \| H\|_{L^1}; ~H \text{ satisfies } \eqref{defa0}\}.
\end{align*}
It is not hard to see that this expression is indeed a norm.

There is an immediate connection between $\mathcal W_t$ and operators on $\mathbb F_t$:
\begin{prop}\label{estimatenorm2}
Given $k \in \mathcal W_t$, the integral operator
\begin{align*}
A_kf(z) := \int_{\mathbb C^n} f(w) k(w,z)~d\mu_t(w)
\end{align*}
is a bounded linear operator on $\mathbb F_t$ for any $\mathbb F_t$ with
\begin{align*}
\| A_k\|_{\mathbb F_t\to \mathbb F_t} \leq  \| k\|_{\mathcal W_t}
\end{align*}
\end{prop}
\begin{proof}
Lemma \ref{normestimate} shows that $A_k$ is a bounded operator on $L_t^p$ satisfying the required norm estimate. We therefore only need to show that $A_k$ leaves $\mathbb F_t$ invariant. For this, fix $f \in F_t^\infty$ and let $\Delta \subset \mathbb C$ be a triangle. We then have:
\begin{align*}
\int_{\Delta} \int_{\mathbb C^n} |f(w) k(w,z)|d\mu_t(w) dz_k &\lesssim \int_{\Delta} \int_{\mathbb C^n} |f(w)| H(w-z) e^{-\frac{|w|^2}{2t} + \frac{|z|^2}{2t}} ~dw~dz_k\\
&\leq \| f\|_{F_t^\infty} \int_{\Delta} e^{\frac{|z|}{2t}} \int_{\mathbb C^n} H(w-z) ~dw~dz_k\\
&= \| f\|_{F_t^\infty} \int_{\Delta} e^{\frac{|z|^2}{2t}} \int_{\mathbb C^n} H(w-z)~dw~dz_k\\
&= \| f\|_{F_t^\infty} \| H\|_{L^1} \int_{\Delta} e^{\frac{|z|^2}{2t}}~dz_k\\
&< \infty.
\end{align*}
Therefore, we may apply Fubini's theorem in the following computation to exchange the order of integration:
\begin{align*}
\int_{\Delta} A_k (f)(z)~dz_k &= \int_{\Delta} \int_{\mathbb C^n} f(w) k(w,z) e^{-\frac{|w|^2}{t}}~dw~dz_k\\
&= \int_{\mathbb C^n} f(w) e^{-\frac{|w|^2}{2t}} \int_\Delta k(w,z) ~dz_k ~dw\\
&= 0,
\end{align*}
where the inner integral equals zero in the last step by Cauchy's Theorem. Therefore, by Morera's Theorem, $A_k(f)$ is holomorphic in every variable, hence it is holomorphic. In particular, $A_k$ leaves every $F_t^p$ invariant. From this it follows easily that $A_k$ leaves $f_t^\infty$ invariant.
\end{proof}
In principle, the space $\mathcal W_t$ consists essentially of well-behaved Berezin transforms:
\begin{prop}\label{prop:comp_Berezin}
Let $k \in \mathcal W_t$. Then,
\begin{align*}
\widetilde{A_k}(w,z) = e^{-\frac{|w|^2 + |z|^2}{2t}} k(w,z).
\end{align*}
\end{prop}
\begin{proof}
Note that
\begin{align*}
\int_{\mathbb C^n} |k(w,z)|d\mu_{2t}(w) \leq e^{\frac{|z|^2}{2t}} \int_{\mathbb C^n} H(w-z)~dw < \infty,
\end{align*}
hence $k(\cdot, z) \in L_t^1$ and similarly $k(w, \cdot) \in L_t^1$. Therefore, $\overline{k(\cdot, z)}, k(w, \cdot) \in F_t^1$ and we obtain the reproducing formulas
\begin{align*}
\langle k(\cdot, z), \overline{K_w^t}\rangle = k(w, z) = \langle k(w, \cdot), K_z^t\rangle.
\end{align*}
This yields:
\begin{align*}
e^{\frac{|w|^2 + |z|^2}{2t}}\widetilde{A_k}(w,z) &= \langle A_k K_w^t, K_z^t\rangle\\
&= \int_{\mathbb C^n} \int_{\mathbb C^n} k(u,v) K_w^t(u)~d\mu_t(u) \overline{K_z^t(v)}~d\mu_t(v)\\
&= \int_{\mathbb C^n} \langle k(\cdot, v), \overline{K_w^t}\rangle \overline{K_z^t(v)}~d\mu_t(v)\\
&= \int_{\mathbb C^n} k(w, v) \overline{K_z^t(v)}~d\mu_t(v)\\
&= \langle k(w, \cdot), K_z^t\rangle\\
&= k(w,z).\qedhere
\end{align*}
\end{proof}
We endow $\mathcal W_t$ with the product
\begin{align*}
k_1 \circ k_2(w,z) = \int_{\mathbb C^n} k_1(w, \xi) k_2(\xi, z)~d\mu_t(\xi),
\end{align*}
and the involution
\begin{align*}
k^\ast (w, z) := \overline{k(z,w)}.
\end{align*}
Of course, the product $\circ$ is compatible with the composition of operators and their Berezin transform: For $A, B \in \mathcal L(F_t^2)$, it is
\begin{align*}
\widetilde{AB}(w, z) &= \langle AB k_w^t, k_z^t\rangle\\
&= e^{-\frac{|z|^2+|w|^2}{2t}} \langle B K_w^t, A^\ast K_z^t\rangle\\
&= e^{-\frac{|z|^2+|w|^2}{2t}} \int_{\mathbb C^n} BK_w^t(\xi) \overline{A^\ast K_z^t(\xi)}~d\mu_t(\xi)\\
&= e^{-\frac{|z|^2+|w|^2}{2t}} \int_{\mathbb C^n} \langle B K_w^t, K_\xi^t\rangle \langle A K_\xi^t, K_z^t\rangle ~d\mu_t(\xi).
\end{align*}
\begin{thm}\label{thm:algebra}
Endowed with above product and norm, $\mathcal W_t$ becomes an involutive Banach algebra.
\end{thm}
\begin{proof}
Clearly, $\mathcal W_t$ is a linear space. We show that the product maps again into $\mathcal W_t$: We have
\begin{align*}
|k_1 \circ k_2(w,z)| &\leq \frac{1}{(\pi t)^n} e^{\frac{|z|^2 + |w|^2}{2t}} \int_{\mathbb C^n} H_1(w-\xi) H_2(\xi - w)~d\xi\\
&= \frac{1}{(\pi t)^n} e^{\frac{|z|^2 + |w|^2}{2t}} \int_{\mathbb C^n} H_1(\xi) H_2(w - z - \xi)~d\xi\\
&= \frac{1}{(\pi t)^n} e^{\frac{|z|^2 + |w|^2}{2t}}H_1 \ast H_2 (w-z).
\end{align*}
Since
\begin{align*}
\| H_1 \ast H_2\|_{L^1} \leq \| H_1\|_{L^1} \| H_2\|_{L^1},
\end{align*}
we obtain that $k_1 \circ k_2 \in \mathcal W_t$ can be estimated by an appropriate $L^1$ function. Taking the infimum over all such $H_1, H_2$, we obtain that the norm is submultiplicative. $k_1 \circ k_2(w, z)$ is of course also holomorphic in $z$ and anti-holomorphic in $w$, as
\begin{align*}
k_1 \circ k_2(w, z) = \langle A_{k_1} A_{k_2} K_w^t, K_z^t\rangle,
\end{align*} 
which is holomorphic in $z$ and anti-holomorphic in $w$.

Finally, we show that $\mathcal W_t$ is complete with respect to the norm. Let $(k_m)_{m\in \mathbb N}$ be a sequence in $\mathcal W_t$ such that
\begin{align*}
\sum_{m=1}^\infty \| k_m\|_{\mathcal W_t} < \infty.
\end{align*}
Then, by the norm estimate in Lemma \ref{normestimate}, we obtain (say, over $F_t^2$)
\begin{align*}
\sum_{m=1}^\infty \| A_{k_m}\| < \infty.
\end{align*}
Since $\mathcal L(F_t^2)$ is complete, we get
\begin{align*}
\sum_{m=1}^\infty A_{k_m} =: A \in \mathcal L(F_t^2).
\end{align*}
Set now $k(w,z) := \langle A K_w^t, K_z^t\rangle$. Since the series of operators converges in operator norm, it is by Proposition \ref{prop:comp_Berezin}
\begin{align*}
k(w,z) &= \sum_{m=1}^\infty \langle A_{k_m} K_w^t, K_z^t\rangle\\
&= \sum_{m=1}^\infty k_m(w,z)
\end{align*}
and the series converges pointwise. Now, we can choose $H_m \in L^1(\mathbb C^n)$ non-negative such that
\begin{align*}
|k_m(w,z)| \leq e^{\frac{|z|^2 + |w|^2}{2t}} H_m(w-z)
\end{align*}
and
\begin{align}\label{eqhn}
\sum_{m=1}^\infty \| H_m\|_{L^1} < \infty.
\end{align}
Thus,
\begin{align*}
\left | k(w,z)\right | &\leq \sum_{m=1}^\infty  |k_m(w,z)|\\
&\leq e^{\frac{|z|^2 + |w|^2}{2t}} \sum_{m=1}^\infty H_m(w-z).
\end{align*}
Due to \eqref{eqhn}, the right-hand side of the above estimates is finite almost everywhere and satisfies
\begin{align*}
\| k\|_{\mathcal W_t} \leq \sum_{m=1}^\infty \| H_m\|_{L^1} < \infty.
\end{align*}
This shows $\sum_{m=1}^\infty k_m = k \in \mathcal W_t$.
\end{proof}
Let us denote
\begin{align*}
\mathcal W_t(\mathbb F_t) := \{ A \in \mathcal L(\mathbb F_t); ~A = A_k, ~\text{some } k \in \mathcal W_t\}.
\end{align*}
Since for $A, B \in \mathcal W_t(\mathbb F_t)$, the operator $AB$ is naturally the integral operator with kernel $k_1 \circ k_2$ (where $k_1, k_2$ are the kernels of $A, B$ respectively), $\mathcal W_t$ can be considered as a subalgebra of $\mathcal L(\mathbb F_t)$. The estimate from Proposition \ref{estimatenorm2} shows:
\begin{prop}
The map
\begin{align*}
\mathcal W_t &\to \mathcal L(\mathbb F_t)\\
k &\mapsto A_k
\end{align*}
is a continuous and injective Banach algebra homomorphism with range $\mathcal W_t(\mathbb F_t)$.
\end{prop}
Note that operators from $\mathcal W_t(\mathbb F_t)$ are weakly localized in the sense of \cite{Isralowitz_Mitkovski_Wick}. In particular, $\mathcal W_t(\mathbb F_t)$ is properly contained in $\mathcal L(\mathbb F_t)$ (e.g., $Uf(z) = f(-z)$ is not contained in the Banach algebra generated by all weakly localized operators). 

As in \cite{Fulsche2020, Fulsche2022} we can consider the Banach algebra
\begin{align*}
    \mathcal C_1(\mathbb F_t) := \{ A \in \mathcal L(\mathbb F_t): ~\| \alpha_z(A) - A\|_{op} \to 0, ~|z| \to 0\}.
\end{align*}
Here, we used the notation $\alpha_z(A)$ for the shift of the operator $A$:
\begin{align*}
    \alpha_z(A) := W_z^t A W_{-z}^t.
\end{align*}
It is known (cf.\ \cite[Theorem 3.1]{Fulsche2020} or \cite[Theorem 4.1]{Fulsche2022}) that
\begin{align*}
\mathcal C_1(\mathbb F_t) =  \overline{Alg} \{ T_f^t \in \mathcal L(\mathbb F_t): ~f \in L^\infty(\mathbb C^n)\}.
\end{align*}
Here, $\overline{Alg}$ denotes the Banach algebra generated by the set. We have the following:
\begin{thm}
It is $\mathcal W_t(\mathbb F_t) \subset \mathcal C_1(\mathbb F_t)$ for every choice of $\mathbb F_t$. 
\end{thm}
\begin{proof}
 R.\ Hagger shows in \cite{Hagger2021} that the Banach algebra generated by the weakly localized operators coincides with $\mathcal C_1(\mathbb F_t)$ whenever $1 < p < \infty$. His proof works in entirely the same way over $F_t^1$, cf.\ the discussion preceding \cite[Theorem 6.2]{Fulsche2022}. Since $\mathcal W_t(\mathbb F_t)$ sits between the sufficiently localized and the weakly localized operators, the claim follows. For $\mathbb F_t = F_t^\infty$ or $\mathbb F_t = f_t^\infty$, the statement follows by duality.
\end{proof}
Since every Toeplitz operator with bounded symbol is sufficiently localized \cite[Proposition 4.1]{Xia_Zheng2013}, the sums of products of Toeplitz operators with bounded symbols are contained in $\mathcal W_t(\mathbb F_t)$ for every $\mathbb F_t$.

Let us start with the investigation of independence of spectral properties of $A_k$, $k \in \mathcal W_t$, of the particular choice of $\mathbb F_t$. The first statement of this kind is regarding compactness:
\begin{thm}
Let $k \in \mathcal W_t$. Then, compactness of $A_k$ is independent of the choice of $\mathbb F_t$.
\end{thm}
\begin{proof}
    Since $A_k \in \mathcal C_1(\mathbb F_t)$, we know that $A_k \in \mathcal K(\mathbb F_t)$ if and only if $\widetilde{A_k}(z, z) = k(z,z)e^{-\frac{|z|^2}{t}} \in C_0(\mathbb C^n)$ (cf.\ \cite[Theorem 1.1]{Bauer_Isralowitz2012}, \cite[Proposition 3.5]{Fulsche2020} or \cite[Theorem 4.4]{Fulsche2022}). Since this last property is clearly independent of the choice of $\mathbb F_t$, the statement follows.
\end{proof}

The following result is at the heart of everything that follows:
\begin{thm}\label{thm:inverse_closed}
$\mathcal W_t(\mathbb F_t)$ is inverse-closed in $\mathcal L(\mathbb F_t)$ for any choice of $\mathbb F_t$.
\end{thm}
\begin{proof}
    We recall that the Fock spaces $F_t^p(\mathbb C^n)$ are the isometric images of the modulation spaces $M^p(\mathbb R^n)$ under the Bargmann transform $\mathcal B_t$. 
    By \cite[Theorem 3.2]{Grochenig2006}, the condition $A \in \mathcal W_t(\mathbb F_t)$ is equivalent to $\mathcal B_t^{-1} A\mathcal B_t = \mathrm{op}^w(\sigma)$ with $\sigma \in M^{\infty, 1}(\mathbb R^{2n})$, where $\mathrm{op}^w$ denotes the Weyl quantization. $\mathrm{op}^w(M^{\infty, 1}(\mathbb R^{2n}))$ is spectrally invariant in each of the spaces $\mathcal L(M^p(\mathbb R^n))$, and invertibility on one of the $M^p(\mathbb R^n)$ implies invertibility on each of the modulation spaces. This can be deduced from \cite[Corollary 4.7]{Grochenig2006}, but the precise statement can be found in \cite[Theorem 1.2]{Grochenig_Pfeuffer_Toft2022}.
\end{proof}
At this point, we will focus on applications of the result.

To emphasize that our integral operators in question act on different Banach spaces, we will denote by $\sigma(A: X \to X)$ the spectrum of the linear operator $A \in \mathcal L(X)$.
\begin{cor}
Let $k \in \mathcal W_t$. Then, $\sigma(A_k: \mathbb F_t \to \mathbb F_t)$ is independent of the particular choice of $\mathbb F_t$.
\end{cor}

For $A \in \mathcal L(\mathbb F_t)$ we denote $\alpha_z(A) = W_z^t A W_{-z}^t$. Then, for $\mathbb F_t \neq F_t^\infty$ and $A \in \mathcal C_1(\mathbb F_t)$, 
\begin{align*}
\mathbb C^n \ni z \mapsto \alpha_z(A) \in \mathcal C_1(\mathbb F_t)
\end{align*}
extends to a map
\begin{align*}
\mathcal M(\operatorname{BUC}(\mathbb C^n)) \ni x \mapsto \alpha_x(A) \in \mathcal C_1(\mathbb F_t),
\end{align*}
which is continuous with respect to the strong operator topology, cf. \cite{Bauer_Isralowitz2012, Fulsche_Hagger, Fulsche2022}. Here, $\mathcal M(\operatorname{BUC}(\mathbb C^n))$ denotes the maximal ideal space of the unital $C^\ast$ algebra of bounded uniformly continuous functions on $\mathbb C^n$ (with respect to the Euclidean metric). We consider this maximal ideal space as a compactification of $\mathbb C^n$. The operators $\alpha_x(A)$ for $x \in \partial \mathbb C^n := \mathcal M(\operatorname{BUC}(\mathbb C^n)) \setminus \mathbb C$ are called \emph{limit operators of $A$}. For $A \in \mathcal C_1(F_t^\infty)$, we define the limit operators through duality, i.e.\ $\alpha_x(A) := (\alpha_x(A|_{f_t^\infty}))^{\ast \ast}$. In the following, we use the notation $\sigma_{ess}(A: X \to X)$ for the essential spectrum considered over the Banach space $X$. Recall that the essential spectrum for an operator from $\mathcal C_1(\mathbb F_t)$ can be computed through the spectra of its limit operators:
\begin{thm}[{\cite{Fulsche_Hagger, Fulsche2022}}]\label{thm:ess_spec}
Let $A \in \mathcal C_1(\mathbb F_t)$. Then, $A$ is Fredholm if and only if $\alpha_x(A)$ is invertible for every $x \in \partial \mathbb C^n$. In particular,
\begin{align*}
\sigma_{ess}(A: \mathbb F_t \to \mathbb F_t) = \bigcup_{x \in \partial \mathbb C^n} \sigma(\alpha_x(A): \mathbb F_t \to \mathbb F_t).
\end{align*}
\end{thm}
We have the following important fact concerning the Wiener algebra and limit operators:
\begin{lem}\label{lemma:WienerLimit}
Let $A \in \mathcal W_t(\mathbb F_t)$. Then, $\alpha_x(A) \in \mathcal W_t(\mathbb F_t)$ for every $x \in \partial \mathbb C^n$.
\end{lem}
\begin{proof}
Let $k \in \mathcal W_t$ such that $A = A_k$ and $H \in L^1(\mathbb C^n)$ as in the definition. Direct computations, using the identity $W_z^t W_w^t = e^{-\frac{\operatorname{Im}(z \cdot \overline w)}{t}}W_{z+w}^t$, easily show that
\begin{align*}
\langle \alpha_v(A_k) K_w^t, K_z^t\rangle = e^{\frac{|z|^2+|w|^2}{2t}} e^{i \frac{\im (v\cdot \overline{(w-z)})}{t}} \langle A k_{w-v}^t, k_{z-v}^t\rangle.
\end{align*}
Since $A_k \in \mathcal C_1(F_t^2)$, $\alpha_{v_\gamma}(A_k) \to \alpha_x(A_k)$ in strong operator topology whenever $v_\gamma \to x \in \partial \mathbb C^n$. In particular,
\begin{align*}
\langle \alpha_{v_\gamma}(A_k) k_w^t, k_z^t\rangle \to \langle \alpha_x(A_k) k_w^t, k_z^t\rangle
\end{align*}
for every $w, z\in \mathbb C^n$. Hence, 
\begin{align*}
|\langle \alpha_x(A_k) k_w^t, k_z^t\rangle| &= \lim_{\gamma} |\langle \alpha_{v_\gamma}(A_k) k_w^t, k_z^t\rangle| \\
&= \lim_{\gamma} | \langle A_k k_{w-v_\gamma}^t, k_{z-v_\gamma}^t\rangle|\\
&\leq \limsup_{\gamma} H((w-v_\gamma) - (z-v_\gamma))\\
&= H(w-z).
\end{align*}
This shows that $k_x(w,z) := \langle \alpha_x(A_k) K_w^t, K_z^t\rangle \in \mathcal W_t$, hence $\alpha_x(A_k) \in \mathcal W_t(\mathbb F_t)$.
\end{proof}
Combining the previous lemma with Theorems \ref{thm:ess_spec} and \ref{thm:inverse_closed} yields:
\begin{cor}\label{cor:FredholmnessIndependent}
Let $k \in \mathcal W_t$. Then, $\sigma_{ess}(A_k: \mathbb F_t \to \mathbb F_t)$ is independent of the particular choice of $\mathbb F_t$.
\end{cor}

Having obtained the previous result, one will of course wonder if the Fredholm index depends on the choice of $\mathbb F_t$. Indeed, this is not the case:
\begin{thm}\label{thm:IndexConst}
    Let $k \in \mathcal W_t$ such that $A_k$ is Fredholm. Then, $\operatorname{ind}(A_k)$ does not depend on the choice of $\mathbb F_t$.
\end{thm}
Proving this result hinges on the theory of complex interpolation. We have the following general result on the complex interpolation method:
\begin{thm}[{\cite[Theorem 4.3]{Albrecht1984}}]
    Let $\overline{X}$ be a compatible couple of Banach spaces and $A \in \mathcal L(\overline{X})$. Then, 
    \begin{align*}
        F_A := \{ \theta \in [0, 1]: ~A: \overline{X}_{[\theta]} \to \overline{X}_{[\theta]} \text{ is Fredholm}\}
    \end{align*}
    is an open subset of $[0, 1]$ and $\operatorname{ind}(A: \overline{X}_{[\theta]} \to \overline{X}_{[\theta]})$ is a continuous function of $\theta \in F_A$.
\end{thm}
We want to mention that the endpoints $\theta = 0, 1$ are not included in the formulation of \cite[Theorem 4.3]{Albrecht1984}, but there is no problem in extending the result to the endpoints by a straightforward modification of \cite[Lemma 4.2]{Albrecht1984}.
\begin{proof}[Proof of Theorem \ref{thm:IndexConst}]
    Let $k \in \mathcal W_t$. With the compatible couple $\overline{X} = (F_t^1, f_t^\infty)$ we clearly have $A_k \in \mathcal L(\overline{X})$. If we assume that $A_k$ is Fredholm, then, with the notation of the previous theorem, $F_{A_k} = [0, 1]$ by Corollary \ref{cor:FredholmnessIndependent}. By the continuity of the index in the previous theorem, we obtain that the Fredholm index is the same on every $F_t^p$, $1 \leq p < \infty$, and on $f_t^\infty$. Finally, the Fredholm index on $F_t^\infty$ equals that on $f_t^\infty$, as $A_k: F_t^\infty \to F_t^\infty$ is the double adjoint of $A_k: f_t^\infty \to f_t^\infty$.
\end{proof}

\section{Discussion}
We have seen that for $k \in \mathcal W_t$, none of $\sigma(A_k)$, $\sigma_{ess}(A_k)$ and $\operatorname{ind}(A_k)$ depend on the choice of the space $\mathbb F_t$ on which the operator is realized. Of course, there is much more spectral data of interest, and one could wonder about $\mathbb F_t$ (in)dependence of any kind of spectral quantity. Just to give one example, one could ask if the multiplicity of eigenvalues is $\mathbb F_t$-independent. Maybe the most pressing question in this respect is the following:
\begin{quest}
    Let $k \in \mathcal W_t$. Is the property $A_k \in \mathcal N(\mathbb F_t)$, the class of nuclear operators, independent of $\mathbb F_t$? Is $\tr(A_k)$ independent of $\mathbb F_t$? Is it always $\tr(A_k) = \int_{\mathbb C^n} k(z,z)d\mu_t(z)$?
\end{quest}
We also want to mention that we are certain that the results described in this work extend to the case of vector-valued Fock spaces $\mathbb F_t^N$. As is well-known, this is crucial for obtaining a satisfactory index theory whenever $n > 1$. The algebra $\mathcal W_t^{N \times N}$ is then easily defined by a coordinate-wise condition, and basic properties of $\mathcal W_t^{N \times N}$ are easily described as for $N = 1$. Among the three important results that led to our spectral invariance properties, the interpolation results from \cite{Albrecht1984} apply in the same way to the vector-valued case. The limit operator techniques from \cite{Fulsche_Hagger, Fulsche2022} can also be extended to the vector-valued case, involving only straightforward adaptations (see also the general setting in \cite{Hagger_Seifert2020}, which covers the vector-valued case in the reflexive range). We are also certain that the spectral invariance properties of the algebra $\operatorname{op}^w(M^{\infty, 1}(\mathbb R^{2n}))$ from \cite{Grochenig2006, Grochenig_Pfeuffer_Toft2022}, that we made crucial use of, extends to the vector-valued case. Nevertheless, it seems inappropriate to state this as a fact without adding some details on it, so we defer from formulating the spectral invariances in the vector-valued case as a theorem.

Of course, one can ask analogous questions of spectral invariance about operators on Bergman spaces. The most natural setting seems to be that of the complex ball in $\mathbb C^n$ or, more generally, of bounded symmetric domains. For simplicity, we shall restrict our discussion to the disk $\mathbb D = \{ z \in \mathbb C: ~|z| < 1\}$. The Bergman space $A_\alpha^p(\mathbb D)$ is usually defined as
\begin{align*}
    A_\alpha^p(\mathbb D) := \operatorname{Hol}(\mathbb D) \cap L^p(\mathbb D, \nu_\alpha),
\end{align*}
where $\alpha > -1$ and $\nu_\alpha$ is the following probability measure on $\mathbb D$:
\begin{align*}
    d\nu_\alpha(z) = \frac{\Gamma(\alpha + 2)}{\Gamma(\alpha + 1)}(1-|z|^2)^\alpha ~dz
\end{align*}
For $p = 2$, the reproducing kernel is now given by $K_z^\alpha(w) = (1-w \overline{z})^{-(\alpha + 2)}$. The unfavourable outcome of these conventions is the fact that $\| K_z^\alpha\|_{A_\alpha^p} = (1-|z|^2)^{-\frac{2+\alpha}{q}}$, where $\frac{1}{p} + \frac{1}{q}= 1$. This means that, for $A \in \mathcal L(A_\alpha^p)$ and the normalized reproducing kernels $k_z^{\alpha, p} = \frac{K_z^\alpha}{\| K_z^\alpha\|_{A_\alpha^p}}$:
\begin{align*}
    \langle A k_z^{\alpha, p}, k_w^{\alpha, p}\rangle = \langle A K_z^\alpha, K_w^\alpha\rangle (1-|z|^2)^{\frac{\alpha+2}{q}} (1-|w|^2)^{\frac{\alpha + 2}{p}}.
\end{align*}
Hence, with these standard conventions there is a natural $p$-dependence in the integral kernels. While this might be overcome with appropriate techniques, the simpler approach would probably be considering the spaces
\begin{align*}
    \mathcal A_\alpha^p(\mathbb D) := \operatorname{Hol}(\mathbb D) \cap L^p(\mathbb D, \nu_{\frac{p(\alpha + 2)}{2}-2}).
\end{align*}
Note that this is exactly in analogy with the scaling of the measure that is done on the Fock space. Upon doing this, one enforces that $\| K_z^\alpha\|_{\mathcal A_\alpha^p} = (1-|z|^2)^{-\frac{2+\alpha}{2}}$, hence the $p$-dependence of the bivariate Berezin transform disappears. One could now consider the class $\mathcal W_\alpha$ of all kernels $k = k(z, w)$, holomorphic in $z$ and anti-holomorphic in $w$, such that
\begin{align*}
    |k(z,w)| \leq (1-|z|^2)^{-\frac{2+\alpha}{2}}(1-|w|^2)^{-\frac{2+\alpha}{2}} H(z-w),
\end{align*}
where $H \in L^1(\mathbb D, \mu)$. Here, $\mu$ is the M\"{o}bius-invariant measure on $\mathbb D$, $d\mu(z) = (1-|z|^2)^{-2}dz$. It seems not unreasonable to expect that operators $A \in \mathcal L(\mathcal A_\alpha^p)$ with integral kernels in this class have similar invariance properties to the operators described in this paper. We want to note that results on essential spectra through limit operators are available in \cite{Hagger2017, Hagger2019, Hagger_Seifert2020} and the interpolation result from \cite{Albrecht1984} can of course be used in the same way. It is the spectral invariance properties from \cite{Grochenig_Pfeuffer_Toft2022} that needs a substitute in the case of the Bergman space.
\begin{quest}
    Is there a Wiener algebra with similar properties on the Bergman spaces $\mathcal A_\alpha^p$?
\end{quest}

\bibliographystyle{amsplain}
\bibliography{References}

\providecommand{\bysame}{\leavevmode\hbox to3em{\hrulefill}\thinspace}
\providecommand{\MR}{\relax\ifhmode\unskip\space\fi MR }
\providecommand{\MRhref}[2]{%
  \href{http://www.ams.org/mathscinet-getitem?mr=#1}{#2}
}
\providecommand{\href}[2]{#2}
\begin{thebibliography}{10}

\bibitem{Al-Qabani_Virtanen}
A.~Al-Qabani and J.~Virtanen, \emph{{F}redholm theory of {T}oeplitz operators
  on standard weighted {F}ock spaces}, {A}nn. {A}cad. {S}ci. {F}enn.-M.
  \textbf{43} (2018), 769--783.

\bibitem{Albrecht1984}
E.~Albrecht, \emph{Spectral interpolation}, Spectral Theory of Linear Operators
  and Related Topics (Timi\c{s}oara/Herculane, 1983), Operator Theory: Advances
  and Applications, vol.~14, Birkh\"auser, 1984, pp.~13--37.

\bibitem{Bauer_Fulsche2019}
W.~Bauer and R.~Fulsche, \emph{{The resolvent algebra in the Segal-Bargmann
  representation}}, In preparation.

\bibitem{Bauer_Isralowitz2012}
W.~Bauer and J.~Isralowitz, \emph{{C}ompactness characterization of operators
  in the {T}oeplitz algebra of the {F}ock space {$F_\alpha^p$}}, J. Funct.
  Anal. \textbf{263} (2012), 1323--1355.

\bibitem{Bergh_Lofstrom1976}
J.~Bergh and J.~L\"{o}fstr\"{o}m, \emph{{Interpolation Spaces: An
  Introduction}}, Die Grundlehren der mathematischen Wissenschaften in
  Einzeldarstellungen, vol. 223, Springer Verlag, 1976.

\bibitem{Englis1999}
M.~Engli\v{s}, \emph{Compact {T}oeplitz operators via the {B}erezin transform
  on bounded symmetric domains}, Integr. Equ. Oper. Theory \textbf{33} (1999),
  426--455.

\bibitem{Folland1989}
G.~B. Folland, \emph{{Harmonic analysis in phase space}}, Princeton University
  Press, 1989.

\bibitem{Fulsche2022}
R.~Fulsche, \emph{{Toeplitz operators on non-reflexive Fock spaces}}, preprint
  available on arXiv:2202.11440.

\bibitem{Fulsche2020}
\bysame, \emph{{Correspondence theory on $p$-Fock spaces with applications to
  Toeplitz algebras}}, J. Funct. Anal. \textbf{279} (2020), 108661.

\bibitem{Fulsche_Hagger}
R.~Fulsche and R.~Hagger, \emph{{Fredholmness of Toeplitz Operators on the Fock
  Space}}, Complex Anal. Oper. Theory \textbf{13} (2019), 375--403.

\bibitem{Grochenig2006}
K.~Gr\"{o}chenig, \emph{{Time-Frequency Analysis of Sj\"ostrand's class}}, Rev.
  Mat. Iberoamericana \textbf{22} (2006), 703--722.

\bibitem{Grochenig_Pfeuffer_Toft2022}
K.~Gr\"ochenig, C.~Pfeuffer, and J.~Toft, \emph{{Spectral invariance of
  quasi-Banach algebras for matrices and pseudodifferential operators}},
  preprint available at arXiv:2211.13934.

\bibitem{Hagger2017}
R.~Hagger, \emph{{The Essential Spectrum of Toeplitz Operators on the Unit
  Ball}}, Integr. Equ. Oper. Theory \textbf{89} (2017), 519--556.

\bibitem{Hagger2019}
\bysame, \emph{{Limit operators, compactness and essential spectra on bounded
  symmetric domains}}, J. Math. Anal. Appl. \textbf{470} (2019), 470--499.

\bibitem{Hagger2021}
\bysame, \emph{{Essential commutants and characterizations of the Toeplitz
  algebra}}, J. Operator Theory \textbf{86} (2021), 125--143.

\bibitem{Hagger_Seifert2020}
R.~Hagger and C.~Seifert, \emph{Limit operators techniques on general metric
  measure spaces of bounded geometry}, J. Math. Anal. Appl. \textbf{489}
  (2020), 124180.

\bibitem{Isralowitz_Mitkovski_Wick}
J.~Isralowitz, M.~Mitkovski, and B.~D. Wick, \emph{{Localization and
  compactness in Bergman and Fock spaces}}, Indiana Univ. Math. J. \textbf{64}
  (2015), 1553--1573.

\bibitem{Janson_Peetre_Rochberg1987}
S.~Janson, J.~Peetre, and R.~Rochberg, \emph{{H}ankel {F}orms and the {F}ock
  {S}pace}, Rev. Mat. Iberoamericana \textbf{3} (1987), no.~1, 61--138.

\bibitem{Signahl_Toft2011}
M.~Signahl and J.~Toft, \emph{{Remarks on mapping properties for the Bargmann
  transform on modulation spaces}}, Integral Transforms Spec. Funct.
  \textbf{22} (2011), 359–366.

\bibitem{Tien_Khoi2019}
P.~T. Tien and L.~H. Khoi, \emph{{Weighted composition operators between
  different Fock spaces}}, Potential Anal. \textbf{50} (2019), 171–195.

\bibitem{Xia_Zheng2013}
J.~Xia and D.~Zheng, \emph{{Localization and Berezin transform on the Fock
  space}}, J. Funct. Anal. \textbf{264} (2013), 97--117.

\bibitem{Zhu2012}
K.~Zhu, \emph{Analysis on {F}ock {S}paces}, Graduate Texts in Mathematics, vol.
  263, Springer US, New York, 2012.

\end{thebibliography}

\bigskip

\noindent
Robert Fulsche\\
\href{fulsche@math.uni-hannover.de}{\Letter fulsche@math.uni-hannover.de}
\\

\noindent
Institut f\"{u}r Analysis\\
Leibniz Universit\"at Hannover\\
Welfengarten 1\\
30167 Hannover\\
GERMANY

% ------------------------------------------------------------------------
\end{document}